\documentclass[12pt]{amsart}
\usepackage{amssymb}
\usepackage{amsthm}
\usepackage{mathrsfs}
\usepackage{amsfonts}
\usepackage[english]{babel}
\usepackage[T1]{fontenc}
\usepackage[latin1]{inputenc}
\usepackage{fullpage}
\usepackage{amsmath}
\usepackage[all]{xy}
\usepackage{stmaryrd}
\usepackage{hyperref}
\usepackage{cleveref}
\usepackage{color}

\newtheorem{theorem}{Theorem}[section]
\newtheorem{lemma}[theorem]{Lemma}
\newtheorem{corollary}[theorem]{Corollary}
\theoremstyle{definition}
\newtheorem{definition}[theorem]{Definition}

\newtheorem{proposition}[theorem]{Proposition}
\newtheorem{conjecture}[theorem]{Conjecture}

\newtheorem{ttheorem}{Main Theorem}

\theoremstyle{remark}
\newtheorem{remark}[theorem]{Remark}

\def\dvg{dv_{\bar{g}}}
\newcommand{\bg}{\bar{g}}

\numberwithin{equation}{section}

\newcommand{\oc}[1]{\overset{\circ}{#1}}
\newcommand{\quot}[4]{\int_M \frac{\sigma_{#1}(#3)}{\sigma_{#2}(#3)} dv_{#4}}
\newcommand{\Vol}{\mathrm{Vol}}

\newcommand{\Ric}{\mathrm{Ric}}

\begin{document}

\title{Comparison of total quotient curvature}

\author{Jiaqi Chen}
\address{School of Electrical Engineering and Automation, Xiamen University of Technology, Xiamen 361024, Fujian, P.R. China}
\email{chenjiaqi@xmut.edu.cn}

\author{Yi Fang}
\address{Department of Mathematics, Anhui University of Technology, Ma'anshan 243002, Anhui, P.R.China}
\email{flxy85@163.com}

\author{Jingyang Zhong}
\address{School of Mathematics and Statistics, Fuzhou University, Fuzhou 350108, Fujian, P.R. China}
\email{jyzhong@fzu.edu.cn}

\thanks{Jiaqi Chen is partially supported by the Scientific Research Foundation of Xiamen University of Technology Under Grant (No.YKJ23009R) and the Scientific Research Foundation for Young and Middle aged Teachers in Fujian Province Under Grant (No.JAT231105). Yi Fang is partially supported by National Natural Science
Foundation of China (No.12571065).}


\begin{abstract}
In this paper, we establish some comparison theorems for the total quotient curvature. Specifically, we examine the  behavior of the functional with respect to the  total quotient curvature and prove that the background Einstein metric achieves a sharp bound on the total quotient curvature. We prove that if the quotient curvature satisfies a point-wise lower (or upper) bound relative to the Einstein metric, then the corresponding integral inequality holds. Also we can show characterize the equality case.  Our result generalizes the volume comparison theorem for scalar curvature and the rigidity results for $\sigma_k$-curvature. 
\end{abstract}

\maketitle

\section{Introduction}\label{sec:introduction}

The interplay between curvature constraints and the global topology of Riemannian manifolds remains one of the most active frontiers in geometric analysis. The classical Bishop-Gromov's volume comparison theorem provides a fundamental framework for studying manifolds with lower Ricci curvature bounds. In recent decades, the Schouten tensor had emerged as a central object of study, particularly in conformal geometry \cite{Viaclovsky-schouten}. Its special invariants, the $\sigma_k$-curvatures, offer a natural hierarchy of geometric conditions that bridge the gap between scalar and sectional curvature. While the $\sigma_k$-Yamabe problem and the rigidity of space forms have been extensively studied, comparison theorems for these fully nonlinear invariants, specifically regarding volume and total curvature functionals, remain less developed \cite{Alice-Gursky-Paul-fully-nonlinear, Viaclovsky-sigma}.

Recent progress has focused on local stability results near canonical metrics. In this context, variational techniques have been proven effective in establishing rigidity for metrics that is close to Einstein metrics in topology. This paper extends this line of inquiry to the quotient curvatures $\frac{\sigma_k}{\sigma_{k-1}}$. These quotients are of particular analytic interest as they correspond to the elementary symmetric functions of the eigenvalues of the Hessian in the context of Monge-Amp\'ere type equations and appear naturally in geometric flows that preserve convexity \cite{Nirenberg-Dirichlet,Trudinger-Dirichlet}.

Let $(M^n, g)$ be a smooth closed Riemannian manifold of dimension $n \geq 3$. We denote the Riemannian curvature tensor by $\mathrm{Rm}_g$, the Ricci curvature tensor by $\mathrm{Ric}_g$, the scalar curvature by $R_g$, and the Schouten curvature tensor by
\begin{align*}
  S_g =\mathrm{Ric}_g - \frac{R_g}{2(n-1)}g.
\end{align*}
Let $\lambda_1, \dots, \lambda_n$ be the eigenvalues of $S_g$, then for $1 \leq k \leq n$, the $\sigma_k$-curvature is defined as the $k$-th elementary symmetric polynomial of these eigenvalues, i.e. 
\begin{equation*}
  \sigma_k(g) = \sum_{1 \leq i_1 < i_2 < \cdots < i_k \leq n} \lambda_{i_1}\lambda_{i_2}\cdots\lambda_{i_k}.
\end{equation*}
In particular,
\begin{equation*}
  \sigma_1(g) = \mathrm{tr}(S_g) = \frac{n-2}{2(n-1)}R_g,
\end{equation*}
and
\begin{equation*}
  \sigma_2(g) = \frac{1}{2}\left( \sigma_1^2(g) - |S_g|^2 \right).
\end{equation*}
So $\sigma_k$-curvatures can be considered as the generalization of the scalar curvatures. The study of $\sigma_k$-curvature and related fully nonlinear curvature equations such as the $\sigma_k$-Yamabe problem \cite{Viaclovsky-sigma} has attracted considerable interest in conformal geometry and PDEs communities.

The quotient curvature arises naturally in conformal geometry, which are defined by
\begin{equation}\label{def: quotient}
    \frac{\sigma_k(g)}{\sigma_l(g)}, \qquad 0 \leq l < k \leq n,
\end{equation}
where $\sigma_0(g) = 1$. And  $\frac{\sigma_k(g)}{\sigma_{k-1}(g)}$ plays the role of a fully nonlinear analogue of the scalar curvature and appears in the Euler-Lagrange equations of certain conformally invariant functionals \cite{GuanPengfei-sigma}. In geometric flows, quotient curvatures govern fully nonlinear curvature flows that generalize mean curvature flow and Gauss curvature flow which are scale-invariant and  convexity conditions preserved related to the G\aa rding cones. 

The classical volume comparison theorem states that if
\begin{equation*}
    \mathrm{Ric}_g \geq (n-1)g,
\end{equation*}
then
\begin{equation*}
    \mathrm{Vol}_{M}(g) \leq \mathrm{Vol}_{\mathbb{S}^n}(g_{\mathbb{S}^n}),
\end{equation*}
where $\mathbb{S}^n$ is the unit round sphere and $g_{\mathbb{S}^n}$ is the canonical metric. 

For closed hyperbolic manifolds, Schoen \cite{Schoen-conjecture} proposed the following conjecture.
\begin{conjecture}
    Let $(M^n, \bg)$ be any closed hyperbolic manifold of dimension $n \geq 3$. Then for any metric $g$ on $M^n$ with
    \begin{equation*}
        R_g \geq R_{\bg},
    \end{equation*}
    the following volume comparison holds:
    \begin{equation*}
        \mathrm{Vol}_{M}(g) \geq \mathrm{Vol}_{M}(\bg).
    \end{equation*}
\end{conjecture}

Schoen's conjecture was proved affirmatively for the 3-dimensional case due to the works of Hamilton \cite{Hamilton-non-singular} and Perelman \cite{Perelman1,Perelman2}. For higher dimensional cases, Besson, Courtois, and Gallot \cite{BCG1, BCG2} verified it for metrics $g$ satisfying $||g-\bg||_{C^2(M^n,\bg)} < \varepsilon_0$ for some small constant $\varepsilon_0 > 0$, i.e., $g$ is $C^2$-close to $\bg$. They also showed that the same result holds without assuming $C^2$-closeness if the assumption on scalar curvature is replaced by Ricci curvature.

If $g$ is $C^2$-close to some stable Einstein metric $\bg$ (see Definition ??), Yuan \cite{Yuan-scalar} proved the following volume comparison theorem with respect to scalar curvature. 

\begin{theorem}[\cite{Yuan-scalar}]\label{Theorem: Yuan}
    Suppose $(M^n,\bg)$ is a strictly stable Einstein manifold with
    \begin{align*}
      \Ric_{\bg} = (n-1)\lambda\bg.
    \end{align*}
    Then there exists a constant $\varepsilon_0>0$ such that for any metric $g$ on $M^n$ satisfying
    \begin{equation*}
    R_g \geq R_{\bg}, \quad ||g - \bg||_{C^2(M^n,\bg)}<\varepsilon_0,
    \end{equation*}
    the following volume comparisons hold:
    \begin{enumerate}
      \item if $\lambda>0$, then $\Vol_{M}(g) \leq \Vol_M(\bg)$;
      \item if $\lambda<0$, then $\Vol_{M}(g) \geq \Vol_M(\bg)$,
    \end{enumerate}
    with equality holding in either case if and only if $g$ is isometric to $\bg$.
\end{theorem}

Following this volume comparison theorem, Lin and Yuan \cite{Yuan-Lin-1, Yuan-Lin-2} proved the similar results in the $Q$-curvature setting. Then Andrade et al.\cite{Andrade-sigma-2} and Chen et al. \cite{Chen-sigma-k} considered in the setting of $\sigma_2$-curvature and $\sigma_k$-curvature ($k\geq2$), respectively. 

Motivated by these results on the volume comparison theorems on scalar type curvatures, in this paper, we consider the comparison theorem with respect to the total quotient curvatures.

\begin{ttheorem}\label{theorem: main1}
    Let $(M^n, \bar{g})$ be a strictly stable Einstein manifold with
    \begin{equation*}
        \Ric_{\bar{g}} = (n-1)\lambda \bar{g},
    \end{equation*}
    where $\lambda > 0$ is a constant. Given positive integers $1\leq l < k \leq n$ and $1 \leq q \leq p \leq n$, there exists a constant $\varepsilon_0 > 0$ such that for any metric $g$ on $M^n$ satisfying
    \begin{equation*}
        \|g - \bar{g}\|_{C^2(M^n, \bar{g})} < \varepsilon_0,
    \end{equation*}
    if one of the following assumptions holds,
    \begin{equation*}
        \begin{split}
            (1) \ \ \frac{\sigma_k(g)}{\sigma_l(g)}  \geq  \frac{\sigma_k(\bar{g})}{\sigma_l(\bar{g})}, \ \ 2(p-q) < n; \ \ (2) \ \ \frac{\sigma_k(g)}{\sigma_l(g)}  \leq  \frac{\sigma_k(\bar{g})}{\sigma_l(\bar{g})}, \ \ 2(p-q) \geq n + 2(k-l),
        \end{split}
    \end{equation*}
    then we have the total quotient curvature comparison
    \begin{equation*}
        \int_M \frac{\sigma_p(g)}{\sigma_q(g)} \, dv_g \leq \int_M \frac{\sigma_p(\bar{g})}{\sigma_q(\bar{g})} \, dv_{\bar{g}}.
    \end{equation*}
 Moreover, equality holds if and only if $g$ is isometric to $\bar{g}$.
\end{ttheorem}

\begin{remark}
    \Cref{theorem: main1} is the generalization of several comparison results. The case $p = q$ corresponds to the volume comparison theorem. Note that $\sigma_0(g) = 1$, then for the specific choice $l=p=q=0$, it reduces to the volume comparison theorems for $\sigma_k$-curvature in \cite{Chen-sigma-k}, and further reduces to the volume comparison theorem with respect to the scalar curvature when $k=1$ \cite{Yuan-scalar}.
    Moreover, the case $l=q=0$ yields a comparison for the total $\sigma_p$-curvature under the $\sigma_k$ lower bound which is asserted in \cite{Chen-total-sigma-k}.
\end{remark}

\begin{remark}
    The admissibility condition for the indices $p$ and $q$ depends on the parity of the dimension $n$, particularly in the case where the constraints on the background curvature are minimal, i.e. $k-l=1$. If $n$ is even, since $p-q\neq n/2$ due to the definition of the key functional we defined in \Cref{definition of functional}, then assumption (1) and (2) cover all possible values of $p-q$. However, if $n$ is odd, assumption (1) and (2) can't cover the case $p-q = \frac{n+1}{2}$ since the second order variation of the key functional will change sign.
\end{remark}
\begin{corollary} \label{thm: gap_result}
    Let $(M^n, \bar{g})$ be a strictly stable Einstein manifold with
    \begin{equation*}
        \Ric_{\bar{g}} = (n-1)\lambda \bar{g},
    \end{equation*}
    where $\lambda > 0$ is a constant. Given positive integers $1\leq l < k \leq n$ and $1 \leq q \leq p \leq n$, there exists a constant $\varepsilon_0 > 0$ such that for any metric $g$ on $M^n$ satisfying
    \begin{equation*}
        \|g - \bar{g}\|_{C^2(M^n, \  \bar{g})} < \varepsilon_0,
    \end{equation*}
    if one of the following assumption holds,
    \begin{equation*}
        \begin{split}
            (1') \ \ \frac{\sigma_k(g)}{\sigma_{k-1}(g)}  \geq  \frac{\sigma_k(\bar{g})}{\sigma_{k-1}(\bar{g})}, \ \ 2(p-q) < n; \ \ (2') \ \ \frac{\sigma_k(g)}{\sigma_{k-1}(g)}  \leq  \frac{\sigma_k(\bar{g})}{\sigma_{k-1}(\bar{g})}, \ \ 2(p-q) \geq n + 2,
        \end{split}
    \end{equation*}
    then we have the total quotient curvature comparison
    \begin{equation*}
        \int_M \frac{\sigma_p(g)}{\sigma_q(g)} \, dv_g \leq \int_M \frac{\sigma_p(\bar{g})}{\sigma_q(\bar{g})} \, dv_{\bar{g}}.
    \end{equation*}
    Moreover, the equality holds if and only if $g$ is isometric to $\bar{g}$.
\end{corollary}


Notably, even restricting to the volume case, our result provides sharper quantitative control. For instance, if we consider the case $k=1$ (where the assumption is $R_g \geq R_{\bg}$) and apply the comparison for indices $p=1, q=0$, we obtain
\begin{equation*}
    \int_M R_g \, dv_g \leq \int_M R_{\bg} \, dv_{\bg}.
\end{equation*}
If we further assume that $R_g$ is a constant, this inequality implies
\begin{equation*}
    R_g \Vol_M(g) \leq R_{\bg} \Vol_M(\bg),
\end{equation*}
which yields the explicit ratio bound:
\begin{equation*}
    \frac{\Vol_{M}(g)}{\Vol_M(\bg)} \leq \frac{R_{\bg}}{R_{g}}.
\end{equation*}
Since $R_g \geq R_{\bg}$, this is a strictly stronger estimate than the standard volume comparison $\Vol_{M}(g) \leq \Vol_M(\bg)$ provided by \Cref{Theorem: Yuan} when the scalar curvature is strictly increasing.

This paper is organized as follows. In \Cref{sec:preliminary}, we introduce some necessary notation and preliminaries that will be used throughout this paper. In \Cref{sec:functional}, we calculate the first and second order variational formulas for the quotient curvature functionals. In \Cref{sec:main}, we prove the main theorems. 

\vskip 0.1 in
{\bf Acknowledgments.} The authors would like to thank Professor Wei Yuan for suggesting this problem and the referees for useful comments. 

\section{Preliminary}\label{sec:preliminary}

\subsection{Notations}

Given a Riemannian manifold $(M^n,g)$, we adopt the following convention for Riemann and Ricci curvature tensors
\begin{align*}
  \mathrm{Rm}(X,Y,Z,W) =  R(X, Y)Z\cdot W ,
\end{align*}
where
\begin{align*}
  R(X, Y)Z = \nabla_X \nabla_Y Z - \nabla_Y \nabla_X Z - \nabla_{[X, Y]}Z.
\end{align*}
And in the local coordinate system, for $1\leq i,j,k,l\leq n$, we denote Riemann and Ricci curvature tensors by
\begin{align*}
  R_{ijkl} &= \mathrm{Rm}(\partial_i, \partial_j, \partial_k, \partial_l) = g_{ml}R_{ijk}^m,\\
  R_{jk} &= g^{il}R_{ijkl},
\end{align*}
the Schouten tensor by
\begin{align*}
  S_{ij}=R_{ij}-\frac{R_g}{2(n-1)}g_{ij},
\end{align*}
and the Einstein operator by
\begin{equation*}
    \Delta_E^{g} = \Delta_{g}+ 2\mathrm{Rm}_{g}    
\end{equation*}
where $\Delta_{g}$ is the Laplace-Beltrami operator.

\subsection{Stable Einstein metric} Let $S_2(M^n,\bg)$ be the space of symmetric 2-tensors on $(M^n,\bg)$, and we denote it by $S_2(M)$ for short, then 
\begin{align*}
      S^{TT}_{2, \bg}(M):=\{ h\in S_2(M)| \delta_{\bg} h=0, tr_{\bg}h=0\}
    \end{align*}
is called the space of transverse-traceless symmetric $2$-tensor on $(M^n, \bg)$. Here $\delta_{\bg}$ is the divergence operator defined by$$(\delta_{\bg}h)_i:=\nabla_{\bg}^jh_{ij}.$$
\begin{definition}
    For $n \geq 3$, suppose $(M^n, \bg)$ is a closed Einstein manifold with
    \begin{align*}
      \mathrm{Ric}_{\bg} = (n-1) \lambda \bg.
    \end{align*}
Then $\bg$ is called stable if $\Delta_E^{\bg}$ is a non-positive operator on $S^{TT}_{2, \bg}(M) \setminus \{0\}$, i.e.
\begin{align*}
  \inf_{h \in S^{TT}_{2, \bg}(M) \setminus \{0\}} \frac{-\int_M  h\cdot\Delta_E^{\bg} h dv_{\bar{g}}}{\int_M |h|^2 dv_{\bar{g}}} \geq 0.
\end{align*}
Moreover, $\bg$ is called strictly stable if $\Delta_E^{\bg}$ is a strictly negative operator on $S^{TT}_{2, \bg}(M) \setminus \{0\}$, i.e.
\begin{align*}
 \inf_{h \in S^{TT}_{2, \bg}(M) \setminus \{0\}} \frac{-\int_M  h\cdot\Delta_E^{\bg} h dv_{\bar{g}}}{\int_M |h|^2 dv_{\bar{g}}} > 0.
\end{align*}
\end{definition}

\subsection{{$\sigma_k$}-curvature}
In order to express the $\sigma_k$-curvature in a compact way, we first introduce the generalized Kronecker delta follows by Reilly's calculation in~\cite{Reilly-1, Reilly-2}.
\begin{definition}
    The generalized Kronecker delta is denoted by
    \begin{align*}
        \delta^{j_1\cdots j_k}_{i_1\cdots i_k}=
            \begin{cases}
            1, ~\text{if } j_1 \cdots j_k \text{ are distinct integers with even permutation of } i_1\cdots i_k;\\
            -1, ~\text{if } j_1\cdots j_k \text{ are distinct integers with odd permutation of } i_1\cdots i_k;\\
            0,~\text{in other cases}.
            \end{cases}
    \end{align*}
\end{definition}
The contraction rule of the generalized Kronecker delta is known as
\begin{align*}
    \delta^{i_1\cdots i_p}_{j_1\cdots j_p}\delta^{j_1\cdots j_k}_{i_1\cdots i_k}=p!\frac{(n-k+p)!}{(n-k)!} \delta^{j_{p+1}\cdots j_k}_{i_{p+1}\cdots i_k},
\end{align*}
where $1 \leq p\leq k-1$ and $2\leq k\leq n$. For the special case  $p=1$, we have
\begin{align*}
  \delta^{i_1}_{j_1}\delta^{j_1\cdots j_k}_{i_1\cdots i_k}=(n-k+1) \delta^{j_{2}\cdots j_k}_{i_{2}\cdots i_k}.
\end{align*}
By the induction argument, it holds
\begin{align*}
  \delta^{i_1}_{j_1}\cdots\delta^{i_p}_{j_p}\delta^{j_1\cdots j_k}_{i_1\cdots i_k}=(n-k+1)\cdots (n-k+p) \delta^{j_{p+1}\cdots j_k}_{i_{p+1}\cdots i_k}=\frac{(n-k+p)!}{(n-k)!}\delta^{j_{p+1}\cdots j_k}_{i_{p+1}\cdots i_k}.
\end{align*}

Now, we provide an equivalent definition of $\sigma_k$-curvature as follows.

\begin{definition}
    The $\sigma_k$-curvature is defined as
      \begin{align*}
        \sigma_k(g)=\frac{1}{k!}\delta^{j_1\cdots j_k}_{i_1\cdots i_k}S^{i_1}_{j_1}\cdots S^{i_k}_{j_k}.
      \end{align*}
\end{definition}
If $(M^n, \bg)$ is an Einstein manifold of $n$-dimension with $\mathrm{Ric}_{\bar{g}}=(n-1)\lambda \bar{g}$, we have
\begin{align*}
    S_{\bg}=\frac{n-2}{2} \lambda \bg, \quad \quad  \sigma_k(\bg) = \left(\frac{n-2}{2} \lambda \right)^k \binom{n}{k},
\end{align*}
where $\binom{n}{k}=\frac{n!}{(n-k)!k!}$ $(1\leq k\leq n)$ is the combination number.

\subsection{Basic variational formulas}
We shall calculate the first and second variations for volume, scalar curvature and $\sigma_k$-curvature which will be used in the later sections. Throughout this paper, we assume that $(M^n, \bg)$ is a closed $n$-dimensional Einstein manifold satisfying $\mathrm{Ric}_{\bg} = (n-1) \lambda \bg$ with $\lambda>0$. Under these assumptions, we derive the following basic variational formulas at $\bg$ in the direction of $h$.

The first and second variations of the volume functional are denoted respectively by
\begin{equation}\label{first volume variation}
    (\mathrm{Vol})'(\bg) := D\mathrm{Vol}(\bg)\cdot h=\frac{1}{2}\int_Mtr_{\bg}h \ dv_{\bg}
\end{equation}
and
\begin{equation}\label{second volume variation}
    \begin{split}
        (\mathrm{Vol})''(\bg) := D^2\mathrm{Vol}(\bg)(h,h) = \frac{n-2}{4n}\int_M(tr_{\bg}h)^2dv_{\bg}-\frac{1}{2}\int_M| \overset{\circ}{h}|^2dv_{\bg},
    \end{split}
\end{equation}
where $\overset{\circ}{h}=h-\frac{tr_{\bg}h}{n}\bg$.
The first variation of scalar curvature is
\begin{equation}
    R'_{\bg} :=DR(\bg)\cdot h = -\Delta_{\bg}(tr_{\bg}h)+\delta^2_{\bg}h-(n-1)\lambda tr_{\bg}h.
\end{equation}
The first and second variations of the $\sigma_k$-curvature are
\begin{equation}\label{first sigma-k variation}
   \sigma_k'(\bg) := D\sigma_k(\bg)\cdot h = \frac{k}{n(n-1)\lambda} \binom{n}{k} \left(\frac{n-2}{2}\lambda\right)^k R'_{\bg}= \frac{k}{n(n-1)\lambda} \sigma_k(\bg) R'_{\bg}
\end{equation}
and
\begin{equation}\label{second sigma-k variation}
    \begin{split}
        {\sigma_k}''(\bg)
        =\frac{2k}{n(n-2)\lambda} \sigma_k(\bg) &\bigg[tr_{\bg} S''_{\bg} -\frac{2(k-1)}{(n-1)(n-2) \lambda} | S'_{\bg} |^2 - \frac{2(n-k)}{n-1}h^{ij} S'_{ij}(\bg)\\
        &+\frac{(k-1)(n-2)}{2(n-1)^3 \lambda} (R'_{\bg})^2 + \frac{(2n-k-1)(n-2)}{2(n-1)} \lambda |h|^2\bigg].
    \end{split}
\end{equation}

For the sake of simplicity, we denote by
\begin{equation}
    A_{kl} := \frac{\sigma_k(\bg)}{\sigma_l(\bg)} = \frac{\left( \frac{n-2}{2} \lambda \right)^k \binom{n}{k} }{\left( \frac{n-2}{2} \lambda \right)^l \binom{n}{l} } = \left( \frac{n-2}{2} \lambda \right)^{k-l} \frac{\binom{n}{k}}{\binom{n}{l}},
\end{equation}
\begin{equation}
\left( \frac{\sigma_k}{\sigma_l} \right)'(\bg):=
        D \left( \frac{\sigma_k}{\sigma_l} \right)(\bg) \cdot h , 
\end{equation}
\begin{equation}
 \left( \frac{\sigma_k}{\sigma_l} \right)''(\bg)
            :=D^2 \left( \frac{\sigma_k}{\sigma_l} \right) (\bg)(h,h). 
\end{equation}
With these notations, the first and second order variations of quotient curvature can be deduced as follows.
\begin{proposition}\label{prop: linearization_of_quotient}
    Let $(M^n, \bg)$ be a closed $n$-dimensional Einstein manifold with $\mathrm{Ric}_{\bg} = (n-1) \lambda \bg$, the first variation of the quotient curvature $\sigma_k/\sigma_l$ at $ g=\bar{g}$ in the direction of a symmetric 2-tensor $h$ is given by
    \begin{equation}
        \left( \frac{\sigma_k}{\sigma_l} \right)'(\bg)
        = \frac{k-l}{n(n-1)\lambda} A_{kl} R'_{\bg}.
    \end{equation}
\end{proposition}
\begin{proof}
    By direct calculation, we have
    \begin{equation*}
        \begin{split}
           \left( \frac{\sigma_k}{\sigma_l} \right)'(\bg)
            &= \frac{{\sigma_k}'(\bg)\sigma_l(\bg)-\sigma_k(\bg) {\sigma_l}'(\bg)}{\sigma^2_l(\bg)} 
            =\frac{k-l}{n(n-1)\lambda}\frac{\binom{n}{k}}{\binom{n}{l}}\left(\frac{n-2}{2}\lambda\right)^{k-l} R'_{\bg}.
        \end{split}
    \end{equation*}
    Combining with the notation of $A_{kl}$, the result follows.
\end{proof}
\begin{proposition}\label{prop: second variation of quotient curvature}
    Let $(M^n, \bg)$ be a closed $n$-dimensional Einstein manifold with $\mathrm{Ric}_{\bg} = (n-1) \lambda \bg$, the second order variation of the quotient curvature $\frac{\sigma_k(g)}{\sigma_l(g)}$ at $g=\bar{g}$ in the direction of a symmetric 2-tensor $h$ satisfies
    \begin{equation}
        \begin{split}
        &\left[ \frac{2A_{kl}(k-l)}{n(n-2) \lambda} \right]^{-1}  \left( \frac{\sigma_k}{\sigma_l} \right)''(\bg)\\
        =& tr_{\bg} S''_{\bg} 
        - \frac{2(k+l-1)}{(n-1) (n-2) \lambda} | S'_{\bg} |^2 
        - \frac{2( n - k - l) }{n-1} h^{ij} S'_{ij}(\bg)
        + \frac{(n-2)\left[n (k-l) - (n-2l) \right]}{2 n (n-1)^3 \lambda} (R'_{\bg})^2\\
        & + \frac{(n-2)(2n-k-l-1)}{2(n-1)} \lambda |h|^2,
        \end{split}
    \end{equation}
    where $S'_{\bg}$ and $S''_{\bg}$ denote the first and second order variations of Schouten tensor $S$ respectively.
\end{proposition}

\begin{proof}
   An easy calculation shows that
    \begin{equation}
        \begin{split}
            \left( \frac{\sigma_k}{\sigma_l} \right)''(\bg)
             = \frac{{\sigma_k}''(\bg)}{\sigma_l(\bg)} - A_{kl} \frac{{\sigma_l}''(\bg)}{\sigma_l(\bg)} - \frac{2{\sigma_l}'(\bg) {\sigma_k}'(\bg)}{\sigma_l^2(\bg)} + \frac{2 \left( {\sigma_l}'(\bg) \right)^2 \sigma_k(\bg)}{\sigma_l^3(\bg)}.
        \end{split}
    \end{equation}
   Due to \ref{second sigma-k variation}, the first two terms of the right hand side of the above formula satisfy
    \begin{equation}\label{first two terms in D2}
        \begin{split}
            &\left[ \frac{2A_{kl}(k-l)}{n (n-2) \lambda} \right]^{-1} \left[ \frac{{\sigma_k}''(\bg)}{\sigma_l(\bg)} - A_{kl} \frac{{\sigma_l}''(\bg)}{\sigma_l(\bg)} \right]\\
            =& tr_{\bg} S''_{\bg} 
            - \frac{2(k+l-1)}{(n-1) (n-2) \lambda} |S'_{\bg}|^2 
            - \frac{2( n - k - l) }{n-1} h^{ij} S'_{ij}(\bg)
            + \frac{(n-2)(k+l-1)}{2(n-1)^3 \lambda} (R'_{\bg})^2\\
            & + \frac{(n-2)(2n-k-l-1)}{2(n-1)} \lambda |h|^2,
        \end{split}
    \end{equation}
    and the last two terms satisfy
    \begin{equation}\label{last two terms in D2}
        \begin{split}
            \left[ \frac{2A_{kl}(k-l)}{n (n-2) \lambda} \right]^{-1} \left[ - \frac{2{\sigma_l}'(\bg) {\sigma_k}'(\bg)}{\sigma_l^2(\bg)} + \frac{2 ({\sigma_l}'(\bg))^2 \sigma_k(\bg)}{\sigma_l^3(\bg)} \right]
            =& -\frac{(n-2)l}{n (n-1)^2 \lambda}  (R'_{\bg})^2.
        \end{split}
    \end{equation}
 This completes the proof.
\end{proof}

\section{The Key Functionals and Corresponding Variations}\label{sec:functional}

In this section, we assume that $(M^n, \bar{g})$ is a closed Einstein manifold satisfying
\begin{equation*}
  \mathrm{Ric}_{\bar{g}}=(n-1) \lambda \bar{g},
\end{equation*}
and define the key functional as
\begin{equation}\label{definition of functional}
    \mathcal{H}_{\bar{g}}(g) = \left[ \int_M \frac{\sigma_p(g)}{\sigma_q(g)} dv_g \right]^{2(k-l)} \left[ \int_M \frac{\sigma_k(g)}{\sigma_l(g)} dv_{\bar{g}} \right]^{n-2(p-q)},
\end{equation}
where $1\leq l<k \leq n$ and $1\leq p \leq q \leq n$. And we will derive the first and second order variational formulas for $\mathcal{H}_{\bar{g}}(g)$ at $g=\bg$. Here functionals with a fixed background volume form $dv_{\bar{g}}$, had been studied previously by Fisher-Marsden \cite{Fischer-Marsden} and Yuan \cite{Yuan-Lin-1, Yuan-Lin-2} on the rigidity of scalar curvature and $Q$-curvature, as well as the related volume comparison theorems. 

If we denote by 
\begin{equation}
    \alpha = 2(k-l) \ \ \text{and} \ \ \beta = n - 2(p-q),
\end{equation}
then we can rewrite the key functional as
\begin{equation}\label{definition of functional compact}
    \mathcal{H}_{\bar{g}}(g) = \left[ \int_M \frac{\sigma_p(g)}{\sigma_q(g)} dv_g \right]^{\alpha} \left[ \int_M \frac{\sigma_k(g)}{\sigma_l(g)} dv_{\bar{g}} \right]^{\beta}.
\end{equation}

And we also have
\begin{equation}
    \mathcal{H}_{\bg}(\bg) = A_{pq}^{\alpha} A_{kl}^{\beta} \mathrm{Vol}_{\bg}^{\alpha+\beta}.
\end{equation}

\subsection{First order variation formula}

We now turn our attention to the fundamental properties of the functional $\mathcal{H}_{\bar{g}}(g)$. A natural first inquiry concerns its behavior under basic transformations and the identification of its critical points. Throughout this subsection, we consider the variation formulas near the background Einstein metric $\bar{g}$.

\begin{proposition}\label{proposition:3.1}
    $\mathcal{H}_{\bar{g}}(g)$ is  scaling invariant with respect to metric $g$, i.e.
    \begin{align*}
      \mathcal{H}_{\bar{g}}(c^2 g) = \mathcal{H}_{\bar{g}}(g)
    \end{align*}
    holds for any non-zero constant $c$.
\end{proposition}

\begin{proof} 
Recall that under a metric scaling $g \mapsto c^2 g$, the $k$-th elementary symmetric function of the eigenvalues of the Schouten tensor transforms as 
$$\sigma_k(c^2 g) = c^{-2k} \sigma_k(g),$$ 
and the volume form scales as 
$$dv_{c^2 g} = c^n dv_g,$$
where $n$ is the dimension of the manifold $M$. Then
\begin{align*}
    \left[ \int_M \frac{\sigma_p(c^2 g)}{\sigma_q(c^2 g)} dv_{c^2 g} \right]^{\alpha} = \left[ c^{\beta} \int_M \frac{\sigma_p(g)}{\sigma_q(g)} dv_{g} \right]^{\alpha} \ \ \text{and} \ \ \left[ \int_M \frac{\sigma_k(c^2 g)}{\sigma_l(c^2 g)} dv_{\bg} \right]^{\beta} = \left[ c^{- \alpha} \int_M \frac{\sigma_k(g)}{\sigma_l(g)} dv_{\bg} \right]^{\beta}.
\end{align*}
So the result follows.
\end{proof}

This property implies that the critical points of the key functional $\mathcal{H}_{\bar{g}}$ should be achieved by metrics whose variations preserve certain geometric scales, particularly those related to curvature. To analyze such critical points, one must understand how the the quotient curvature responds to the infinitesimal changes in the metric. The subsequent proposition provides this analysis by computing the linearization of these quotient curvature at an Einstein metric. This variation, expressed purely in terms of the linearization of scalar curvature, reflects the constrained nature of geometric deformations on Einstein manifolds.

Next, we establish the variational characterization of the background metric. The following proposition confirms that the Einstein metric $\bg$ is a critical point of the functional $\mathcal{H}_{\bg}$. This result follows from the first order variation formula and exploits the fact that, on an Einstein manifold, the integral of the linearized scalar curvature is proportional to the integral of the trace of the metric perturbation.

\begin{proposition}\label{proposition:einstein_critical}
  The Einstein metric $\bar{g}$ is the critical point of $\mathcal{H}_{\bar{g}}$, i.e.
  \begin{align*}
      D \mathcal{H}_{\bg} (\bg)\cdot h = 0.
  \end{align*}
  for any $h\in S_2(M)$.
\end{proposition}

\begin{proof}
  Let $g_t=\bar{g}+th$, where $h\in S_2(M)$, $t\in(-\varepsilon,\varepsilon)$ for some small $\varepsilon>0$. Then the linearization of the functional $\mathcal{H}_{\bar{g}}$ at $g=\bar{g}$ is
  \begin{equation}\label{eq: first_variation}
      \begin{split}
          D\mathcal{H}_{\bar{g}}(\bar{g})\cdot h 
          =& \alpha \left[ \quot{p}{q}{\bar{g}}{\bar{g}} \right]^{\alpha-1} \left[ \quot{k}{l}{\bar{g}}{\bg} \right]^{\beta} D \left[ \quot{p}{q}{g}{g} \right] (\bar{g})\cdot h\\
          &+ \beta \left[ \quot{p}{q}{\bar{g}}{\bar{g}} \right]^{\alpha} \left[ \quot{k}{l}{\bar{g}}{\bg} \right]^{\beta-1} D \left[ \quot{k}{l}{g}{\bar{g}} \right](\bar{g})\cdot h
      \end{split}
  \end{equation}
Note that for any $k(1 \leq k \leq n)$, there holds
\begin{align*}
      \sigma_k(\bg)=\bigg(\frac{n-2}{2}\lambda\bigg)^k{\binom{n}{k}},  \ \ {\sigma_k}'(\bg)=\frac{k}{n(n-1)\lambda} \sigma_k(\bg)  R'_{\bg},
  \end{align*}
And integration by parts yields 
\begin{align*}
    \int_MR'_{\bg}dv_{\bg}
    =\int_M[-\Delta_{\bg}tr_{\bg}h+\delta^2_{\bg} h- (n-1) \lambda tr_{\bg} h]dv_{\bg}
    =-(n-1)\lambda \int_M tr_{\bg}h dv_{\bg}.
\end{align*}
Combining this with \Cref{prop: linearization_of_quotient}, it follows 
\begin{equation}\label{eq:1st_differential_of_fix_vol_int}
    D \left[ \quot{k}{l}{g}{\bg} \right](\bar{g})\cdot h = \frac{A_{kl} (k-l)}{n(n-1)\lambda}  \int_M R'_{\bg} dv_{\bg} = - \frac{A_{kl} (k-l)}{n} \int_M tr_{\bg} h dv_{\bg},
\end{equation}
and
\begin{equation}\label{eq:1st_differential_of_int}
    \begin{split}
        D \left[ \quot{p}{q}{g}{g} \right](\bar{g})\cdot h 
        &= - \frac{A_{pq} (p-q)}{n} \int_M tr_{\bg} h dv_{\bg} +  \frac{A_{pq}}{2} \int_M tr_{\bg} h dv_{\bg} \\
        &= \frac{A_{pq} \left[ n - 2(p-q) \right]}{2n} \int_M tr_{\bg} h dv_{\bg}.
    \end{split}
\end{equation}
Therefore, $D\mathcal{H}_{\bg}(\bg)\cdot h = 0$, i.e. $\bg$ is a critical point of $\mathcal{H}_{\bar{g}}$.
\end{proof}

\subsection{Second order variation formula}
In this subsection, we derive the second order variational formula for the quotient curvature functional. To facilitate the computation of the spectrum, we adopt the standard decomposition of the symmetric 2-tensor $h$ into its trace-free transverse (TT) part and its pure trace part. Specifically, we assume the symmetric 2-tensor $h$ admits the TT-gauge decomposition:
\begin{equation}
    h = \oc{h} + \frac{1}{n} (tr_{\bg}h) \bg \in S_{2,\bg}^{TT}(M) \oplus \left( C^{\infty}(M) \cdot \bg\right)
\end{equation}
where $\oc{h} \in S_{2,\bg}^{TT}(M)$ satisfies $tr_{\bg} \oc{h} = 0$ and $\delta_{\bg} \oc{h} = 0$. Under this decomposition, the divergence of $h$ satisfies
\begin{equation}
    \delta_{\bg}(h) = - \frac{1}{n} \nabla_{\bg}(tr_{\bg} h).
\end{equation}
And we also have the following formula
\begin{equation}
    \int_M |h|^2dv_{\bg} = \int_M|\overset{\circ}{h}|^2dv_{\bg} +\frac{1}{n}\int_M(tr_{\bg}h)^2dv_{\bg}
\end{equation}
for the integral decomposition. We recall the following integration identities for the Schouten tensor on Einstein manifolds, as derived in \cite{Chen-sigma-k}.
\begin{proposition}\label{integration formulas}
Let $(M^n, \bg)$ be a closed $n$-dimensional Einstein manifold with $\mathrm{Ric}_{\bg} = (n-1) \lambda \bg$, we have the following integration formulas:
\begin{align*}
    \int_M \ tr_{\bg} S''_{\bg}dv_{\bg}
    =&-\frac{3n-2}{4(n-1)} \int_M \oc{h} \cdot \Delta^{\bg}_E \oc{h} dv_{\bg}
    + \frac{n-2}{2} \lambda \int_M| \oc{h}|^2 dv_{\bg}
    -\frac{(n-2)^2}{4n^2} \int_M |\nabla_{\bg} tr_{\bg}h|^2dv_{\bg};\\
    \int_M |S'_{\bg}|^2 dv_{\bg}
    =& \frac{1}{4} \int_M |\Delta^{\bg}_E \oc{h}|^2 dv_{\bg}
    -\frac{n-2}{2} \lambda \int_M \oc{h} \cdot \Delta^{\bg}_E \oc{h} dv_{\bg}
    +\frac{(n-2)^2}{4} \lambda^2 \int_M |\oc{h}|^2 dv_{\bg}\\
    &+ \frac{(n-2)^2}{4n^2} \int_M (\Delta_{\bg}tr_{\bg}h)^2 dv_{\bg}
    -\frac{(n-1)(n-2)^2}{4n^2} \lambda \int_M |\nabla_{\bg}tr_{\bg}h|^2 dv_{\bg};\\
    \int_M h^{ij} S'_{ij}(\bg) dv_{\bg}
    =&-\frac{1}{2} \int_M \oc{h} \cdot \Delta^{\bg}_E \oc{h} dv_{\bg}
    +\frac{n-2}{2} \lambda\int_M |\oc{h}|^2dv_{\bg}
    +\frac{n-2}{2n^2} \int_M |\nabla_{\bg}tr_{\bg}h|^2 dv_{\bg};\\
    \int_M (R'_{\bg})^2dv_{\bg}
    =& (n-1)^2 \left[\frac{1}{n^2} \int_M (\Delta_{\bg} tr_{\bg}h)^2dv_{\bg}
    -\frac{2}{n} \lambda \int_M |\nabla_{\bg} tr_{\bg}h|^2 dv_{\bg}
    +\lambda^2 \int_M (tr_{\bg}h)^2dv_{\bg} \right];
\end{align*}
where $\Delta^{\bg}_E=\Delta_{\bg}+2Rm_{\bg}$ denotes the Einstein operator acting on symmetric 2-tensors.
\end{proposition}

Combining with \Cref{prop: second variation of quotient curvature,integration formulas}, we obtain the second order variation formula for the key functional $\mathcal{H}_{\bg}$.

\begin{proposition}\label{prop: second order variation of functional}
    Let $(M^n, \bg)$ be a closed $n$-dimensional Einstein manifold with $\mathrm{Ric}_{\bg} = (n-1) \lambda \bg$, and
    \begin{equation*}
        h = \mathring{h} + \frac{tr_{\bg}h}{n} \bg \in S^{TT}_{2,\bg}(M) \oplus(C^\infty(M) \cdot\bg),
    \end{equation*}
   then the second order variation formula for $\mathcal{H}_{\bg}$ at $\bar{g}$ is
    \begin{equation*}
        \begin{split}
            &\left[ \mathcal{H}_{\bg}(\bg) \right]^{-1} \mathrm{Vol}_{\bg} D^2 \mathcal{H}_{\bg}(\bar{g}) (h, h)\\
            =& -\frac{\alpha\left[ \beta(k+l) + 2(p^2 - q^2) -n \right]}{2n(n-1)(n-2)^2 \lambda^2} \int_M | \Delta_{E}^{\bg} \oc{h}|^2 dv_{\bg} + \frac{\alpha}{4(n-1)\lambda} \int_M \oc{h} \cdot \Delta^{\bg}_{E} \oc{h} dv_{\bg}\\
            &- \frac{\alpha \left[ 2(p-q)(q-l)+nl\right]}{n^4 \lambda^2} \int_M \left[(\Delta_{\bg} tr_{\bg} h)^2 - n \lambda |\nabla_{\bg} tr_{\bg} h|^2 \right] dv_{\bg}\\
            & - \frac{\alpha\beta(\alpha+\beta)}{4n^3 \lambda} \int_M \left[|\nabla_{\bg} tr_{\bg} h|^2 - n\lambda \left( tr_{\bg} h - \overline{tr_{\bg} h} \right)^2 \right]dv_{\bg},
        \end{split}        
    \end{equation*}
    where $ \overline{tr_{\bg}h}:= \mathrm{Vol}_{\bg}^{-1} \int_M tr_{\bg}hdv_{\bg}$ denotes the average of $tr_{\bg}h$ on $M^n$, and
        $\alpha = 2(k-l)$, $\beta = n - 2(p-q)$.
  
\end{proposition}

\begin{proof}
 Based on the first order variation formula we derived in \Cref{eq: first_variation}, the second order variation of the functional $\mathcal{H}_{\bar{g}}$ is
    \begin{align*}
        &D^2 \mathcal{H}_{\bg}(\bar{g}) (h, h)\\
        =& \alpha (\alpha-1) \left[ \quot{p}{q}{\bar{g}}{\bar{g}} \right]^{\alpha-2} \left[ \quot{k}{l}{\bar{g}}{\bg} \right]^{\beta} \left\{ D \left[ \quot{p}{q}{g}{g} \right](\bar{g})\cdot h \right\}^2\\
        &+ 2 \alpha \beta \left[ \quot{p}{q}{\bar{g}}{\bar{g}} \right]^{\alpha-1} \left[ \quot{k}{l}{\bar{g}}{\bg} \right]^{\beta-1}\\
        &\times\left\{ D \left[ \quot{p}{q}{g}{g} \right] (\bar{g})\cdot h \right\} \left\{D \left[ \quot{k}{l}{g}{\bg} \right](\bar{g})\cdot h\right\} \\
        &+ \alpha \left[ \quot{p}{q}{\bar{g}}{\bar{g}} \right]^{\alpha-1} \left[ \quot{k}{l}{\bar{g}}{\bg} \right]^{\beta} D^2 \left[ \quot{p}{q}{g}{g} \right](\bar{g})(h, h) \\
        &+ \beta (\beta -1)\left[ \quot{p}{q}{\bar{g}}{\bar{g}} \right]^{\alpha} \left[ \quot{k}{l}{\bar{g}}{\bg} \right]^{\beta-2} \left\{ D \left[ \quot{k}{l}{g}{\bg}\right](\bar{g})\cdot h  \right\}^2\\
        &+ \beta \left[ \quot{p}{q}{\bar{g}}{\bar{g}} \right]^{\alpha} \left[ \quot{k}{l}{\bar{g}}{\bg} \right]^{\beta-1} D^2 \left[ \quot{k}{l}{g}{\bg} \right](\bar{g})(h, h)\\
        =:& J_1 + J_2 + J_3 + J_4 + J_5.
    \end{align*}
    Combining the above formula with \Cref{eq:1st_differential_of_fix_vol_int,eq:1st_differential_of_int}, we have 
    \begin{align*}
        J_1 &= \frac{\alpha (\alpha - 1 ) \beta^2}{4n^2} \mathcal{H}_{\bg}(\bg) \mathrm{Vol}_{\bg}^{-2} \left[\int_M tr_{\bg} h dv_{\bg} \right]^2;\\
        J_2 &= -\frac{\alpha^2 \beta^2}{2n^2} \mathcal{H}_{\bg}(\bg) \mathrm{Vol}_{\bg}^{-2} \left[\int_M tr_{\bg} h dv_{\bg} \right]^2;\\
        J_4 &= \frac{\alpha^2 \beta (\beta- 1)}{4n^2} \mathcal{H}_{\bg}(\bg) \mathrm{Vol}_{\bg}^{-2} \left[\int_M tr_{\bg} h dv_{\bg} \right]^2,
    \end{align*}
    and 
    \begin{align*}
        \left[ \mathcal{H}_{\bg}(\bg) \right]^{-1} \mathrm{Vol}_{\bg} \left[J_1 + J_2 + J_4\right] = -\frac{\alpha \beta (\alpha+\beta)}{4n^2}  \mathrm{Vol}_{\bg}^{-1} \left[\int_M tr_{\bg} h dv_{\bg} \right]^2.
    \end{align*}
    On the other hand,
    \begin{align*}
        J_3 
        &= \alpha \mathcal{H}_{\bg}(\bg) \left[ A_{pq} \mathrm{Vol}_{\bg} \right]^{-1} D^2 \left[ \quot{p}{q}{g}{g} \right](\bar{g})(h, h)\\
        &= \alpha \mathcal{H}_{\bg}(\bg) \left[ A_{pq} \mathrm{Vol}_{\bg} \right]^{-1} \left\{ \int_M  \left(\frac{\sigma_p}{\sigma_q}\right)''(\bar{g}) dv_{\bg} +  2 \int_M  \left(\frac{\sigma_p}{\sigma_q}\right)'(\bar{g})  (dv)'(\bar{g})\right\}\\
        & \ \ \  +\alpha \mathcal{H}_{\bg}(\bg) \left[ A_{pq} \mathrm{Vol}_{\bg} \right]^{-1} \frac{\sigma_p(\bar{g})}{\sigma_q(\bar{g})}\int_M (dv)''(\bar{g})\\
        &= \alpha \mathcal{H}_{\bg}(\bg) \left[ A_{pq} \mathrm{Vol}_{\bg} \right]^{-1} \left\{ \int_M  \left(\frac{\sigma_p}{\sigma_q}\right)''(\bar{g})dv_{\bg} +  \int_M  \left(\frac{\sigma_p}{\sigma_q} \right)'(\bar{g})tr_{\bg}hdv_{\bg} + A_{pq} (\mathrm{Vol})''(\bar{g})\right\}
    \end{align*}
    and
    \begin{align*}
        J_5
        &= \beta \mathcal{H}_{\bg}(\bg) \left[ A_{kl} \mathrm{Vol}_{\bg} \right]^{-1} \int_M  \left(\frac{\sigma_k} {\sigma_l}\right)''(\bar{g})dv_{\bg},
    \end{align*}
    therefore,
    \begin{align*}
        \left[ \mathcal{H}_{\bg}(\bg) \right]^{-1} \mathrm{Vol}_{\bg} \left[J_3 + J_5 \right]
        =& \alpha A_{pq}^{-1} \int_M \left( \frac{\sigma_p}{\sigma_q}\right)''(\bg) dv_{\bg} + \beta A_{kl}^{-1} \int_M \left( \frac{\sigma_k}{\sigma_l}\right)''(\bg) dv_{\bg}\\
        &+ \alpha A_{pq}^{-1} \int_M \left(  \frac{\sigma_p}{\sigma_q}\right)'(\bg)tr_{\bg} h dv_{\bg} + \alpha (\mathrm{Vol})''(\bg).
    \end{align*}
    For the first two terms, by \Cref{prop: second variation of quotient curvature}, we have
    \begin{align*}
        &\alpha A_{pq}^{-1} \int_M \left(\frac{\sigma_p}{\sigma_q}\right)'' (\bg) dv_{\bg} + \beta A_{kl}^{-1} \int_M \left(\frac{\sigma_k}{\sigma_l}\right)''(\bg)dv_{\bg}\\
        =& C_1 \int_M tr_{\bg} S''_{\bg} dv_{\bg} + C_2 \int_M |S'_{\bg}|^2dv_{\bg} + C_3 \int_M h^{ij} S'_{ij}(\bg) dv_{\bg} + C_4  \int_M (R'_{\bg})^2dv_{\bg} + C_5 \int_M |h|^2 dv_{\bg}
    \end{align*}
    where
    \begin{align*}
        C_1 &= \frac{\alpha}{(n-2) \lambda};\\
        C_2 &= \frac{2 \alpha}{n(n-1)(n-2)^2 \lambda^2} \left[ 2(p-q) (k+l-p-q) - n (k+l-1) \right];\\
        C_3 &= -\frac{2 \alpha}{n(n-1)(n-2) \lambda} \left[ 2(p-q) (k+l-p-q) + n(n-k-l) \right];\\
        C_4 &= \frac{\alpha (p-q)}{n^2 (n-1)^3 \lambda^2} \left[ n(p-q-k+l) + 2(q-l)\right] + \frac{\alpha}{2n(n-1)^3 \lambda^2} \left[ n(k-l) - (n-2l) \right];\\
        C_5 &= \frac{\alpha (p-q)}{n(n-1)} \left[ k+ l - p - q\right] + \frac{(n-k-l)\alpha}{2(n-1)} + \frac{\alpha}{2}.
    \end{align*}
    For the third term, integration by part implies that
    \begin{align*}
        &\alpha A_{pq}^{-1} \int_M  \left( \frac{\sigma_p}{\sigma_q}\right)'(\bg) tr_{\bg}h dv_{\bg}\\
        =& \frac{\alpha(p-q)}{n(n-1)\lambda} \int_M [-\Delta_{\bg}  tr_{\bg} h \cdot tr_{\bg} h + \delta^2_{\bg} h \cdot tr_{\bg} h - (n-1)\lambda  (tr_{\bg} h)^2 ] dv_{\bg}\\
        =& \frac{\alpha(p-q)}{n^2 \lambda} \int_M |\nabla_{\bg} tr_{\bg} h|^2 dv_{\bg} - \frac{\alpha(p-q)}{n} \int_M  (tr_{\bg} h)^2 dv_{\bg}.
    \end{align*}
   And for the last term, by the direct calculation, it follows 
    \begin{align*}
        \alpha  (\mathrm{Vol})''(\bg) = \frac{\alpha}{4} \int_M  (tr_{\bg} h)^2 dv_{\bg} - \frac{\alpha}{2} \int_M |h|^2 dv_{\bg}.
    \end{align*}
    Combining with all the above results, we finally get 
    \begin{equation*}
        \begin{split}
            &\left[ \mathcal{H}_{\bg}(\bg) \right]^{-1} \mathrm{Vol}_{\bg} D^2 \mathcal{H}_{\bg}(\bg)(h, h)\\
            =& -\frac{\alpha\left[ \beta(k+l) + 2(p^2 - q^2) -n \right]}{2n(n-1)(n-2)^2 \lambda^2} \int_M  |\Delta_{E}^{\bg}\oc{h}|^2 dv_{\bg} + \frac{\alpha}{4(n-1)\lambda} \int_M \oc{h} \cdot \Delta_{E}^{\bg} \oc{h} dv_{\bg}\\
            &- \frac{\alpha \left[2(p-q)(q-l) + nl\right]}{n^4 \lambda^2} \int_M \left[(\Delta_{\bg} tr_{\bg} h)^2 - n \lambda |\nabla_{\bg} tr_{\bg} h|^2 \right] dv_{\bg}\\
            & - \frac{\alpha\beta(\alpha+\beta)}{4n^3 \lambda} \int_M \left[|\nabla_{\bg} tr_{\bg} h|^2 - n\lambda \left( tr_{\bg} h - \overline{tr_{\bg} h} \right)^2\right] dv_{\bg},
        \end{split}        
    \end{equation*}
    where
    \begin{equation*}
        \overline{tr_{\bg}h}:= \mathrm{Vol}_{\bg}^{-1} \int_M tr_{\bg}hdv_{\bg}
    \end{equation*}
   denotes the average of $tr_{\bg}h$ on $M$ and
    \begin{equation*}
        \int_M (tr_{\bg}h - \overline{tr_{\bg}h})^2 dv_{\bg}
        =\int_M(tr_{\bg}h)^2dv_{\bg} - \mathrm{Vol}_{\bg}^{-1} \left(\int_M tr_{\bg}h dv_{\bg}\right)^2.
    \end{equation*}
\end{proof}

Before proceeding to the next section, we analyze the coefficients appearing in the second variation. We first observe that the quantity $2(p-q)(q-l)+nl$ is non-negative, given that $1 \leq q \leq p \leq n$ and $l \geq 1$. Focusing now on the coefficient of the first term, we establish the following lemma.

\begin{lemma}\label{lemma: integer_inequality}
    Let $n, k, l, p, q$ be integers satisfying $1\leq l<k \leq n$ and $1\leq q \leq p \leq n$. Then 
    \begin{equation*}
        \beta(k+l) + 2(p^2 - q^2) -n \geq 0.
    \end{equation*}
\end{lemma}

\begin{proof}
    Let $x = p-q$ and $y = p+q$, then $\beta=n-2(p-q)=n-2x$ and 
    \begin{equation*}
    \beta(k+l) + 2(p^2 - q^2) -n 
        =(n-2x)(k+l)+ 2xy - n.
    \end{equation*}

    \noindent \textbf{Case 1:} Suppose $n-2x \geq 0$.
    Since $k+l \geq 3$, $y\geq2$, $x\geq0$, we have
    \begin{equation*}
        \begin{split}
            (n-2x)(k+l)+ 2xy - n &\geq 2(n-x)+ 2x(y-2) \\
            &\geq 2x+ 2x(y-2)=2x(y-1)\geq0.
        \end{split}
    \end{equation*}

    \noindent \textbf{Case 2:} Suppose $n - 2x < 0$.
    Since $k+l \leq 2n-1$, $y > x$ and $x\leq n-1$, then
    \begin{equation*}
        \begin{split}
            (n-2x)(k+l)+ 2xy - n &\geq(n-2x)(2n - 1) + 2x^2 - n \\
            &= 2 (n-x) (n-1-x)\geq0.
        \end{split}
    \end{equation*}
    This completes the proof.
\end{proof}

\section{Total quotient curvature comparison}\label{sec:main}

Having derived the explicit formulas for the first and second variations of the functional $\mathcal{H}_{\bar{g}}$ in Section \ref{sec:functional}, we are now equipped to establish the total quotient curvature comparison results. The variational formulas in Proposition \ref{prop: second order variation of functional} reduce the problem of the total quotient curvatures comparison to the studying the sign  of a quadratic form on the space of symmetric 2-tensors.

In this section, we analyze the specific terms arising in the second variation. By decomposing the perturbations $h$ into trace part $u=tr_{\bg}h$ and trace-free parts $\oc{h}$, we apply the strictly stable Einstein condition to control the transverse-traceless components, while the trace components are estimated by the Bochner technique and the Lichnerowicz-Obata theorem. This spectral analysis culminates in the proof of Theorem \ref{theorem: main1}.

\begin{lemma}\label{lemma: bochner}
Let $(M^n, \bg)$ be a closed n-dimensional Einstein manifold with $\mathrm{Ric}_{\bg} = (n-1) \lambda \bg$, we have
\begin{equation}
    \int_M \left[(\Delta_{\bg} tr_{\bg} h)^2 - n \lambda |\nabla_{\bg} tr_{\bg} h|^2 \right] dv_{\bg} \geq 0.
\end{equation}
\end{lemma}

\begin{proof}
By Bochner formula, for any function $u\in C^\infty(M^n)$
\begin{align*}
    \frac{1}{2}\Delta_{\bg}(|\nabla_{\bg} u|^2)
    &=|\nabla^2_{\bg} u|^2+\langle\nabla_{\bg} u , \nabla_{\bg}\Delta_{\bg} u\rangle+Ric(\nabla_{\bg} u, \nabla_{\bg} u)\\
    &=|\nabla^2_{\bg} u|^2+\langle\nabla_{\bg} u, \nabla_{\bg}\Delta_{\bg} u\rangle+(n-1)\lambda|\nabla_{\bg} u|^2.
\end{align*}
Then integrating on $M^n$ implies that
\begin{align*}
    0&=\frac{1}{2}\int_M\Delta_{\bg}(|\nabla_{\bg} u|^2)dv_{\bg}\\
    &=\int_M|\nabla^2_{\bg} u|^2dv_{\bg}-\int_M(\Delta_{\bg} u)^2 dv_{\bg}+(n-1)\lambda\int_M|\nabla_{\bg} u|^2dv_{\bg}\\
    &\geq\frac{1}{n} \int_M(\Delta_{\bg} u)^2dv_{\bg}-\int_M(\Delta_{\bg} u)^2 dv_{\bg} + (n-1)\lambda\int_M|\nabla_{\bg} u|^2dv_{\bg}\\
    &=-\frac{n-1}{n} \int_M\big[(\Delta_{\bg} u)^2-n\lambda|\nabla_{\bg} u|^2\big] dv_{\bg}, 
\end{align*}
where the above inequality holds due to Schwartz inequality
\begin{align*}
    |\nabla^2_{\bg}u|^2 \geq \frac{1}{n} (\Delta_{\bg} u)^2.
\end{align*}
So if we let $u=tr_{\bg}h$, then
\begin{align*}
    \int_M \left[(\Delta_{\bg}tr_{\bg}h)^2-n \lambda |\nabla_{\bg}tr_{\bg}h|^2 \right]dv_{\bg} \geq 0.
\end{align*}
\end{proof}

Similarly, the fourth term in the second variation formula in \Cref{prop: second order variation of functional} compares the Dirichlet energy of the trace to its variance. To control this term, we require a lower bound on the first non-zero eigenvalue of the Laplace-Beltrami operator $\Delta_{\bar{g}}$. We recall the classical Lichnerowicz-Obata theorem, which provides the sharp spectral gap estimate for manifolds with positive Ricci curvature. For a detailed exposition of this result, we refer the reader to  \cite{Lichnerowicz, Obata}.

\begin{lemma}[Lichnerowicz-Obata]\label{lemma:Lichnerowicz-Obata}
  Let $(M^n, \bg)$ be a closed n-dimensional manifold with
  \begin{align*}
    \mathrm{Ric}_{\bar{g}}\geq(n-1)\lambda \bar{g},
  \end{align*}
  where $ \lambda>0$ is a constant. Then for any function $u \in C^\infty(M^n)$ that is not identically constant, $\bar{u}= \mathrm{Vol}_{\bg}^{-1} \int_M u dv_{\bg}$, we have
  \begin{align*}
    \int_M |\nabla_{\bg} u|^2 dv_{\bar{g}} \geq n\lambda\int_M (u - \bar{u})^2 dv_{\bar{g}},
  \end{align*}
  where equality holds if and only if $(M^n,\bar{g})$ is isometric to the round sphere ${\mathbb{S}}^n(r)$ with radius $r=\frac{1}{\sqrt{\lambda}}$ and $u$ is the first eigenfunction of the Laplace-Beltrami operator $\Delta_{\bar{g}}$.
\end{lemma}

By applying \Cref{lemma: bochner} and \Cref{lemma:Lichnerowicz-Obata} to \Cref{prop: second order variation of functional}, we immediately obtain the non-positive definiteness of the second-order variation of $\mathcal{H}_{\bar{g}}(\cdot)$ at the Einstein metric $\bar{g}$ and we arrive:

\begin{proposition}\label{prop: characteristic_second_order}
  Suppose $(M^n,\bar{g})$ is a strictly stable Einstein manifold with Ricci curvature tensor
  \begin{align*}
    \mathrm{Ric}_{\bar{g}}=(n-1)\lambda\bar{g},
  \end{align*}
  where $\lambda > 0$ is a constant, then $\bar{g}$ is a critical point of $\mathcal{H}_{\bar{g}}(\cdot)$ and 
  \begin{align*}
    D^2\mathcal{H}_{\bar{g}}(\bar{g})(h,h)\leq 0
  \end{align*}
  for any $h=\mathring{h}+\frac{1}{n}(tr_{\bar{g}}h)\bar{g}\in S^{TT}_{2,\bar{g}}(M)\oplus(C^{\infty}(M)\cdot \bar{g})$, provided the indices $1 \leq q \leq p \leq n$ satisfy either $2(p-q) \leq n$ or $2(p-q) \geq n + 2(k-l)$. Furthermore, the equality holds if and only if
  \begin{itemize}
    \item $h\in \mathbb{R}\cdot\bar{g}$, and $(M^n ,\bar{g})$ is not isometric to the round sphere up to a rescaling of the metric.
    \item $h\in E_{n\lambda} \cdot\bar{g}$, and $(M^n,\bar{g})$ is isometric to the round sphere $( \mathbb{S}^n(1/\sqrt{\lambda}), g_0)$,
    where
    \begin{align*}
      E_{n\lambda}=\left\{ u\in C^{\infty} (\mathbb{S}^n ( 1/\sqrt{\lambda} ) ) | \Delta_{g_0}u + n\lambda u=0 \right\}
    \end{align*}
    is the space of the first eigenfunctions for the spherical metric $g_0$.
  \end{itemize}
\end{proposition}

\begin{proof}
    According to the second order variation formula in \Cref{prop: second order variation of functional} and combining with \Cref{lemma: integer_inequality,lemma: bochner,lemma:Lichnerowicz-Obata}, $D^2 \mathcal{H}_{\bar{g}}(\bar{g})(h,h)$ is non-positive definite. Furthermore, when $D^2 \mathcal{H}_{\bar{g}}(\bar{g})(h,h)= 0$, we have
      \begin{align*}
        -\frac{\alpha\left[ \beta(k+l) + 2(p^2 - q^2) -n \right]}{2n(n-1)(n-2)^2 \lambda^2} \int_M |\Delta_{E}^{\bg}\oc{h}|^2 dv_{\bg} + \frac{\alpha}{4(n-1)\lambda} \int_M \oc{h} \cdot \Delta_{E}^{\bg} \oc{h} dv_{\bg} = 0,
      \end{align*}
  \begin{align*}
    \int_M \left[|\nabla_{\bg} tr_{\bar{g}}h|^2 - n\lambda (tr_{\bar{g}}h - \overline{tr_{\bar{g}}h})^2 \right]\dvg = 0
  \end{align*}
and
 \begin{align*}
   \int_M \left[ (\Delta_{\bg} tr_{\bg} h)^2 - n \lambda \left| \nabla_{\bg} tr_{\bg} h \right|^2 \right] \dvg=0.
    \end{align*}
  The first equation implies that $\mathring{h}=0$ since the Einstein metric $\bar{g}$ is strictly stable. For the second and third equations, by Lemma \ref{lemma:Lichnerowicz-Obata}, we have
  \begin{itemize}
    \item if $(M^n,\bar{g})$ is not isometric to the round sphere, then $tr_{\bar{g}}h = \overline{tr_{\bar{g}}h}$, so $tr_{\bar{g}}h$ is constant and $h =\frac{tr_{\bar{g}}h}{n} \bar{g}\in \mathbb{R}\cdot \bar{g}$.
    \item if $(M^n,\bar{g})$ is isometric to the round sphere, then $tr_{\bar{g}}h \in E_{n\lambda}$, thus $h =\frac{tr_{\bar{g}}h}{n} \bar{g} \in  E_{n\lambda} \cdot\bar{g}$.
  \end{itemize}
\end{proof}

In order to investigate the local structure of Einstein metrics, we introduce the following slice theorem \cite{Ebin-slice,Viaclovsky-critical}:
\begin{theorem}[Ebin-Palais slice theorem]\label{Theorem: Ebin-Palais slice theorem}
    Suppose $(M^n,\bg)$ is a closed n-dimensional Einstein manifold with $Ric_{\bg}=(n-1)\lambda\bg$,  $\lambda >0$. Let $\mathcal{M}$ be the space of all Riemannian metrics on $M$. There exists a slice $\mathcal{S}_{\bg}$ through $\bg$ in $\mathcal{M}$. That is, For a fixed real number $p>n$, one can find a constant $\epsilon>0$ such that for any metrics $g \in \mathcal{M}$ with $||g-\bg||_{W^{2,p}(M,\bg)} < \epsilon$, there exists a diffeomorphism $\varphi$ with $\varphi^*g\in\mathcal{S}_{\bg}$. Moreover, for a smooth local slice $\mathcal{S}_{\bg}$, we have \begin{itemize}
        \item $T_{\bg}\mathcal{S}_{\bg}=S_{2,\bg}^{TT}(M)\oplus(C^{\infty}(M)\cdot\bg)$ when $(M^n,\bg)$ is not isometric to the round sphere;
        \item $T_{\bg}\mathcal{S}_{\bg}=S_{2,\bg}^{TT}(M)\oplus(E_{n\lambda}^{\perp}\cdot\bg)$ when $(M^n,\bg)$ is isometric to $\mathbb{S}^n(1/\sqrt{\lambda})$,\\
        where $E_{n\lambda}^{\perp}=\left\{\left.u\in C^{\infty}\left(\mathbb{S}^n\left(1/\sqrt{\lambda}\right)\right)\ \right|\int_{\mathbb{S}^n\left(1/\sqrt{\lambda}\right)}uvdv_{\bg}=0,\ \forall v\in E_{n\lambda}\right\}$ 
    \end{itemize}
  and
    \begin{equation*}
        S_2(M)=\{\mathcal{L}_{\bg}(X)|X\in\mathfrak{X}(M)\}\oplus T_{\bg}\mathcal{S}_{\bg},
    \end{equation*}
    where $\mathfrak{X}(M)$ is the space of smooth vector fields on the manifold $M$ and $\mathcal{L}_{\bg}(X)$ is the Lie derivative of the metric $\bg$ with respect to the vector field $X$.
\end{theorem}

Applying the slice theorem, we can restrict $\mathcal{H}_{{\bg}}$ on a local slice $\mathcal{S}_{\bg}$, and denote it by $\mathcal{H}_{{\bg}}^{\mathcal{S}}$. To investigate the behavior of $\mathcal{H}_{{\bg}}^{\mathcal{S}}$, we need the following Morse lemma \cite{Fischer-Marsden}:

\begin{lemma}[Morse lemma]\label{lemma: morse}
  Let $\mathcal{P}$ be a Banach manifold and $F: \mathcal{P} \rightarrow \mathbb{R}$ a $C^2-$function. Suppose that $\mathcal{Q} \subset \mathcal{P}$ is a submanifold satisfies
    \begin{equation*}
        F = 0 \ \  \text{and} \ \ dF = 0  \ \ \text{on}  \ \ \mathcal{Q} 
    \end{equation*}
  and that there is a smooth normal bundle neighborhood of $\mathcal{Q}$ such that if $\mathcal{E}_x$ is the normal complement to $T_x\mathcal{Q}$ in $T_x\mathcal{P}$, then $d^2F(x)$ is weakly negative definite on $\mathcal{E}_x$, that is
    \begin{equation*}
        d^2F(x)(v,v) \leq 0
    \end{equation*}
  with equality only if $v = 0$. Let $\langle \cdot, \cdot \rangle_x$ be a weak Riemannian structure with a smooth connection and assume that $F$ has a smooth $\langle \cdot, \cdot \rangle_x$-gradient $Y(x)$. Assume that 
  \begin{equation*}
      DY(x): \mathcal{E}_x \rightarrow \mathcal{E}_x
  \end{equation*}
  is an isomorphism for $x \in \mathcal{Q}$. Then there is a neighborhood $\mathcal{U}$ of $\mathcal{Q}$ such that $y \in \mathcal{U}$ and $F(y) \geq 0$ implies that $y \in \mathcal{Q}$.
\end{lemma}

Combining \Cref{prop: second order variation of functional} and  \Cref{Theorem: Ebin-Palais slice theorem}, we can find a local slice $\mathcal{S}_{\bar{g}}$ through $\bar{g}$, and we consider $\mathcal{Q}_{\bar{g}}$ as a submanifold of $\mathcal{S}_{\bar{g}}$, where
\begin{align*}
  \mathcal{Q}_{\bar{g}} := \{ c^2 \bar{g} \in \mathcal{S}_{\bar{g}} | c \neq 0 \}.
\end{align*}
Then we obtain the following rigidity result:

\begin{proposition}\label{prop:rigidity}
  Suppose $(M^n,\bar{g})$ is a strictly stable Einstein manifold with Ricci curvature
  \begin{align*}
    \mathrm{Ric}_{\bar{g}}=(n-1)\lambda\bar{g}
  \end{align*}
  where $\lambda>0$ is a constant. Given the indices $1 \leq q \leq p \leq n$ satisfy either $2(p-q) \leq n$ or $2(p-q) \geq n + 2(k-l)$, there is a neighborhood of $\bar{g}$ in the local slice $\mathcal{S}_{\bar{g}}$, denoted by $\mathcal{U}_{\bar{g}}$, such that, for any metric $g\in \mathcal{U}_{\bar{g}} $ satisfying
  \begin{align*}
    \mathcal{H}_{\bar{g}}^{\mathcal{S}}(g) \geq \mathcal{H}_{\bar{g}}^{\mathcal{S}}(\bg) 
  \end{align*}
 it follows that $g=c^2\bar{g}$ for some positive constant $c$.
\end{proposition}

\begin{proof}
    To align with Lemma \ref{lemma: morse} in the Banach space setting, we set $\mathcal{P} = \mathcal{S}_g$ and define 
    \begin{equation}
        F = \mathcal{H}_{\bar{g}}^{\mathcal{S}}(g) - \mathcal{H}_{\bar{g}}^{\mathcal{S}}(\bar{g}).
    \end{equation}
    Since $\mathcal{H}_{\bar{g}}^{\mathcal{S}}(g)$ is scale-invariant, we have
    \begin{equation*}
        F = 0 \ \  \text{and} \ \ dF = 0  \ \ \text{on}  \ \ \mathcal{Q}_{\bg} := \{ c^2 \bar{g} \in \mathcal{P} | c \neq 0 \}.
    \end{equation*}
    The tangent space of $\mathcal{Q}_{\bg}$ is 
    \begin{equation*}
        T_{\bg} \mathcal{Q}_{\bg} = \mathbb{R}\cdot \bg
    \end{equation*}
    and its $L^2$-complement in $T_{\bg} \mathcal{P}$ is given as follow:
    \begin{itemize}
        \item if $\bar{g}$ is not spherical, then
        \begin{align*}
          \mathcal{E}_{\bar{g}} = S^{TT}_{2,\bar{g}}(M) \oplus \bigg\{u \cdot \bar{g} \bigg| u \in C^{\infty}(M), \int_M tr_{\bar{g}} u dv_{\bar{g}}=0\bigg\},
        \end{align*}
        \item if $\bar{g}$ is spherical, then
        \begin{align*}
          \mathcal{E}_{\bar{g}} = S^{TT}_{2,\bar{g}}(M) \oplus \bigg\{u \cdot \bar{g} \bigg| u \in E_{n \lambda}^{\perp}, \int_M tr_{\bar{g}} u dv_{\bar{g}}=0\bigg\}.
        \end{align*}
    \end{itemize}
    Define a weak Riemannian structure on $\mathcal{P}$:
    \begin{equation*}
        \langle\langle h,k\rangle\rangle_g := \int_M \left(\langle\nabla_g h, \nabla_g k\rangle_g + \langle h,k\rangle_g\right) \, dv_g = \int_M \langle(-\Delta_g + 1)h, k\rangle_g \, dv_g
    \end{equation*}
    for any $g \in \mathcal{P}$ and $h, k \in T_g\mathcal{P}$.
    We also define $\Gamma_{k,l,g}^*$ to be the $L^2$ adjoint of $D\left( \frac{\sigma_k(g)}{\sigma_l(g)}\right)$, and we define the function $f$ by the relation $dv_{\bar{g}} = f_g \, dv_g$. With these expressions, recall the first order variation formula in \Cref{eq: first_variation}, we have:
    \begin{equation*}
        \begin{split}
            D \mathcal{H}_{\bg} (g) (h) = \mathcal{H}_{\bg}(g) \bigg\{ 
            &2(k-l) \left( \int_M \frac{\sigma_p(g)}{\sigma_q(g)} dv_g \right)^{-1} D \left( \int_M \frac{\sigma_p(g)}{\sigma_q(g)} dv_g \right) (h) \\
            & + [n - 2(p-q)] \left( \int_M \frac{\sigma_k(g)}{\sigma_l(g)} dv_{\bg} \right)^{-1} D \left( \int_M \frac{\sigma_k(g)}{\sigma_l(g)} dv_{\bg} \right) (h)\bigg\}
        \end{split}
    \end{equation*}
    According to \cite{Ebin-slice}, it has a smooth connection and the $\langle \langle \cdot, \cdot \rangle \rangle_g$-gradient of $F$ is
    \begin{equation*}
        \begin{split}
            Y(g) = P_g \left( -\Delta_g + 1 \right)^{-1}  \bigg\{ 
            &2(k-l) \mathcal{H}_{\bg}(g) \left( \int_M \frac{\sigma_p(g)}{\sigma_q(g)} dv_g \right)^{-1} \left( \Gamma_{p, q, g}(1) + \frac{1}{2} \frac{\sigma_p(g)}{\sigma_q(g)} g\right) \\
            & + [n - 2(p-q)] \mathcal{H}_{\bg}(g) \left( \int_M \frac{\sigma_k(g)}{\sigma_l(g)} dv_{\bg} \right)^{-1} \Gamma_{k,l,g}(f_g) \bigg\}
        \end{split}
    \end{equation*}
    in which $P_g$ is the orthogonal projection to $T_g \mathcal{P}$. It is easy to see that $Y(g)$ is smooth and its linearization at $\bg$ is
    \begin{equation*}
        DY(\bg) \cdot h = P_{\bg} \left( -\Delta_{\bg} + 1 \right)^{-1} \left( D^2 \mathcal{H}_{\bg} ({\bg}) (h, h) \right)
    \end{equation*}
    for any $h \in \mathcal{E}_{\bar{g}}$. Given the indices $0 \leq q \leq p \leq n$ satisfy either $2(p-q) \leq n$ or $2(p-q) \geq n + 2(k-l)$, the second order variation is strictly negative defined on $\mathcal{E}_{\bar{g}}$, thus, $DY(\bg)$ is an isomorphism on $\mathcal{E}_{\bar{g}}$ due to Lemma \ref{lemma: morse}, and there is a neighborhood $\mathcal{U}$ of $\mathcal{Q}$ such that for any $g \in \mathcal{U}$, $g \in \mathcal{Q}$, that is
    \begin{equation*}
        g = c^2 \bg
    \end{equation*}
    for a constant $c$ and $\mathcal{H}_{\bg}^{\mathcal{S}} (g) \geq \mathcal{H}_{\bg}^{\mathcal{S}} (\bg)$.
\end{proof}

Now we can prove the main Theorem A.

\begin{proof}[Proof of \Cref{theorem: main1}]
    Let $\lambda > 0$ and  $g$ be any Riemannian metric on $M$ satisfying $\|g - \bar{g}\|_{C^2(M, \bar{g})}  < \varepsilon_0$. By Proposition \ref{prop:rigidity}, there exists a diffeomorphism $\varphi$ such that $\varphi^*g \in \mathcal{U}_{\bar{g}} \subset \mathcal{S}_{\bar{g}}$. For  simplicity, here we identify $g$ with its pullback $\varphi^*g$ and assume $g \in \mathcal{U}_{\bar{g}}$.
    
    We argue by contradiction. Suppose that one of the assumptions in (1) or (2) holds, but the conclusion fails. Then the reverse total curvature comparison holds
    \begin{equation} \label{eq:contra_assumption}
        \int_M \frac{\sigma_p(g)}{\sigma_q(g)} \, dv_g \geq \int_M \frac{\sigma_p(\bar{g})}{\sigma_q(\bar{g})} \, dv_{\bar{g}}.
    \end{equation}
   
    Note that
    \begin{equation*}
        \begin{split}
            \mathcal{H}_{\bar{g}}(g) 
            = \left[ \int_M \frac{\sigma_p(g)}{\sigma_q(g)} \, dv_{g} \right]^{\alpha} \left[ \int_M \frac{\sigma_k(g)}{\sigma_l(g)} \, dv_{\bar{g}} \right]^{\beta}.
        \end{split}
    \end{equation*}
    and $\beta=n-2(p-q)$, $\alpha=k-l>0$, then by assumption (1) or (2), it follows from \eqref{eq:contra_assumption} that
 \begin{equation*}
        \begin{split}
            \mathcal{H}_{\bar{g}}(g) 
            \geq \left[ \int_M \frac{\sigma_p(\bar{g})}{\sigma_q(\bar{g})} \, dv_{\bar{g}} \right]^{\alpha} \left[ \int_M \frac{\sigma_k(\bar{g})}{\sigma_l(\bar{g})} \, dv_{\bar{g}} \right]^{\beta} 
            = \mathcal{H}_{\bar{g}}(\bar{g}).
        \end{split}
    \end{equation*}
which contradicts with Proposition \ref{prop: characteristic_second_order}, so we have 
\begin{equation*} 
        \int_M \frac{\sigma_p(g)}{\sigma_q(g)} \, dv_g \leq \int_M \frac{\sigma_p(\bar{g})}{\sigma_q(\bar{g})} \, dv_{\bar{g}}.
    \end{equation*}

Furthermore, if the equality holds, by the rigidity result in Proposition \ref{prop:rigidity}, this implies that $g$ must be conformal to $\bar{g}$. Thus, $g = c^2 \bar{g}$ for some positive constant $c$.
    
    We claim that $c=1$. Under the scaling $g = c^2 \bar{g}$, \eqref{eq:contra_assumption} can be rewritten as
    \begin{equation*}
        \int_M \frac{\sigma_p(g)}{\sigma_q(g)} \, dv_g = c^{n-2(p-q)} \int_M \frac{\sigma_p(\bar{g})}{\sigma_q(\bar{g})} \, dv_{\bar{g}} \geq \int_M \frac{\sigma_p(\bar{g})}{\sigma_q(\bar{g})} \, dv_{\bar{g}}>0,
    \end{equation*}
    which implies $c^{n-2(p-q)} \geq 1$. Consequently,
    \begin{equation}
         \ c \geq 1 \ \ \text{if} \ \ 2(p-q) < n; \quad \text{and} \ c \leq 1 \ \ \text{if} \ \ 2(p-q) \geq n + 2(k-l).
    \end{equation}
    On the other hand, it holds
    \begin{equation*}
        \frac{\sigma_k(g)}{\sigma_l(g)} = c^{-2(k-l)} \frac{\sigma_k(\bar{g})}{\sigma_l(\bar{g})}>0.
    \end{equation*}
   By assumption (1) or (2), we have
    \begin{equation*}
         c^{-2(k-l)} \geq 1 \ \ \text{if} \ \ 2(p-q) < n; \quad \text{and} \quad c^{-2(k-l)} \leq 1 \ \ \text{if} \ \ 2(p-q) \geq n + 2(k-l).
    \end{equation*}
    Since $k > l$, then it follows
    \begin{equation}
         \ c \leq 1 \ \ \text{if} \ \ 2(p-q) < n; \quad \text{and} \quad \ c \geq 1 \ \ \text{if} \ \ 2(p-q) \geq n + 2(k-l).
    \end{equation}
    
   Therefore, $c = 1$ and $\varphi^*g =g= \bar{g}$, which completes the proof.
\end{proof}

\begin{remark}
    When $\lambda < 0$, the definite sign of the second variation $D^2\mathcal{H}_{\bar{g}}(\bg)(h,h)$ becomes delicate. In this case, if $k-l$ is even, the first and last terms in the expression for $D^2\mathcal{H}_{\bar{g}}(\bg)(h,h)$ remain negative, but the second term may dominate and compromise the overall negativity. Conversely, if $k-l$ is odd, the second term is negative, but the first and last terms can become positive, undermining the negative definiteness of the functional. Consequently, a volume comparison theorem of the type considered here is not expected to hold in this regime.
\end{remark}

\bibliographystyle{plain} 
\bibliography{references}

\end{document}